\documentclass[12pt]{article}
\usepackage[a4paper,margin=0.8in]{geometry}
\usepackage[colorlinks,citecolor=magenta,linkcolor=black]{hyperref}
\pdfpagewidth=\paperwidth \pdfpageheight=\paperheight
\usepackage{amsfonts,amssymb,amsthm,amsmath,eucal,tabu,url}
\usepackage{pgf}
 \usepackage{array}
 \usepackage{tikz-cd}
 \usepackage{pstricks}
 \usepackage{pstricks-add}
 \usepackage{pgf,tikz}
 \usetikzlibrary{automata}
 \usetikzlibrary{arrows}
 \usepackage{indentfirst}
\usepackage{tabularx} 

\theoremstyle{plain}
\newtheorem{thm}{Theorem}[section]
\newtheorem{theorem}[thm]{Theorem}

\newtheorem{proposition}[thm]{Proposition}

\theoremstyle{definition}
\newtheorem{definition}[thm]{Definition}
\newtheorem{remark}[thm]{Remark}

\newtheorem{problem}[thm]{Problem}

\newtheorem{thevarthm}[thm]{\varthmname}

\newenvironment{varthm*}[1]{\trivlist\item[]{\bf #1.}\it}{\endtrivlist}

\renewcommand\geq{\geqslant}

\renewcommand\leq{\leqslant}

\newcommand\be{\begin{eqnarray*}}
\newcommand\ee{\end{eqnarray*}}

\newcommand\newop[2]{\def#1{\mathop{\rm #2}\nolimits}}
\newop\log{log}
\newop\ord{ord}
\newop\Gal{Gal}
\newop\SL{SL}
\newop\Bl{Bl}
\newop\mult{mult}
\newop\mass{mass}
\newop\div{div}
\newop\codim{codim}
\newop\sing{sing}
\newop\vdim{vdim}
\newop\edim{edim}
\newop\Ass{Ass}
\newop\size{size}
\newop\reg{reg}
\newop\satdeg{satdeg}
\newop\supp{supp}
\newop\Neg{Neg}
\newop\Nef{Nef}
\newop\Nefh{Nef_H}
\newop\Eff{Eff}
\newop\Zar{Zar}
\newop\MB{MB}
\newop\MBxC{MB\mathit{(x,C)}}
\newop\NnB{NnB}
\newop\Bigg{Big}
\newop\Effbar{\overline{\Eff}}

\def\keywordname{{\bfseries Keywords}}%
\def\keywords#1{\par\addvspace\medskipamount{\rightskip=0pt plus1cm
\def\and{\ifhmode\unskip\nobreak\fi\ $\cdot$
}\noindent\keywordname\enspace\ignorespaces#1\par}}
\def\subclassname{{\bfseries Mathematics Subject Classification
(2020)}\enspace}
\def\subclass#1{\par\addvspace\medskipamount{\rightskip=0pt plus1cm
\def\and{\ifhmode\unskip\nobreak\fi\ $\cdot$
}\noindent\subclassname\ignorespaces#1\par}}

\begin{document}
\title{On the nearly freeness of conic-line arrangements with nodes, tacnodes, and ordinary triple points}
\author{Aleksandra Ga\l ecka}
\date{\today}
\maketitle
\thispagestyle{empty}
\begin{abstract}
In the present note we provide a partial classification of nearly free conic-line arrangements in the complex plane having nodes, tacnodes, and ordinary triple points. In this setting, our theoretical bound tells us that the degree of such an arrangement is bounded from above by $12$. We construct examples of nearly free conic-line arrangements having degree $3,4,5,6,7$, and we prove that in degree $10$, $11$, and $12$ there is no such arrangement.
\subclass{nearly free curves, conic-line arrangements}
\end{abstract}

\section{Introduction}
The theory of plane curve arrangements has recently gained a lot of attention among researchers. It is worth recalling some recent papers devoted to rational curve arrangements in the complex projective plane, for instance \cite{PSz,ST,STY}. One of the most interesting open questions that appears in the literature is the so-called numerical Terao's conjecture. There are some variants of this conjecture, but we focus here on the following. Let $\mathcal{C}, \mathcal{C}' \subset \mathbb{P}^{2}_{\mathbb{C}}$ be two reduced curves such that they have the same weak combinatorics, i.e., they have the same number of irreducible components of the same degree, and the same number of singularities of a given topological type. Assume that $\mathcal{C}'$ is free, then $\mathcal{C}$ has to be free. For instance,  Pokora and Dimca in  \cite{DimcaPokora} proved that the numerical Terao's conjecture holds in the class of conic-line arrangements with nodes, tacnodes, and ordinary triple points. On the other hand, it is known that numerical Terao's conjecture does not hold in the class of (triangular) line arrangements \cite{Mar}. From this perspective, it is natural to understand wider classes of curves and hopefully this effort will help us to understand numerical Terao's conjecture better. In the light of the original Terao's conjecture for line arrangements, which tells us that the freeness is determined by the intersection poset, Dimca and Sticlaru in \cite{DS} defined a new class of curves that is called nearly free. Here our aim is to understand nearly free complex conic-line arrangement with nodes, tacnodes, and ordinary triple points. Our motivation comes from aforementioned paper by Dimca and Pokora \cite{DimcaPokora}, where the authors classify all free conic-line arrangement with nodes, tacnodes, and ordinary triple points. Let us briefly present the main outcome of the note. First of all, we observe in Proposition \ref{bound} that if $\mathcal{C}$ is a conic-line arrangement in the complex plane with nodes, tacnodes, and ordinary triple points, then its degree is bounded from above by $12$. Then we start to analyse, case by case, in which degree we can find a nearly free conic-line arrangement -- it turns out that we can find such arrangements in degree $3,4,5,6,7$. However, for the degree $10$, $11$, and $12$ we show that there does not exist any conic-line arrangement in the complex plane with the prescribed above singularities that is nearly free. Based on the above discussion, one needs to decide the existence of nearly free conic-line arrangements in degree $8,9$.

Through the paper we work exclusively over the complex numbers. Our computations are performed with \verb}Singular} \cite{Singular}.

\section{Conic-line arrangements with nodes, tacnodes, and ordinary triple points}
Let $\mathcal{C} \subset \mathbb{P}^{2}_{\mathbb{C}}$ be an arrangement consisting of $d\geq 1$ lines and $k\geq 1$ smooth conics. We assume that $\mathcal{C}$ has only $n_{2}$ nodes, $t$ tacnodes, and $n_{3}$ ordinary triple points. Denote by $m$ the degree of $\mathcal{C} $ which is equal to $m = 2k+d$. Then we have the following combinatorial count
$$4\cdot\binom{k}{2} + 2kd + \binom{d}{2} = \binom{m}{2}-k = n_{2} + 2t + 3n_{3},$$
and this formula follows from B\'ezout theorem. 

Focusing on simple singularities like nodes, tacnodes, and ordinary triple points is the first non-trivial case when non-ordinary singularities occurs. For the discussion we will need also the notion of the global and local Tjurina numbers of singularities.

\begin{definition}
Let $f : (\mathbb{C}^{2}, 0) \rightarrow (\mathbb{C},0)$ be the germ of an isolated singularity. Then the (local) Tjurina number at $p=(0,0)$ is defined as
$$\tau_{p} := {\rm dim}_{\mathbb{C}}\frac{\mathbb{C}\{x,y\}}{\langle f, \partial_{x}\, f, \partial_{y}\, f\rangle}.$$
The total Tjurina number of a reduced curve $C$ is defined as
$$\tau(C) := \sum_{p \in {\rm Sing}(C)} \tau_{p},$$
where the sum goes over all singular points of $C$.
\end{definition}
\begin{remark}
Let us recall that if $p$ is a node, then $\tau_{p}=1$, if $q$ is a tacnode, then $\tau_{q}=3$, and if $r$ is an ordinary triple point, then $\tau_{r} = 4$. Based on it, if $\mathcal{C}$ is a conic-line arrangement with $n_{2}$ nodes, $t$ tacnodes, and $n_{3}$ ordinary triple points, then
$$\tau(\mathcal{C}) = n_{2} + 3t + 4n_{3}.$$
\end{remark}

\section{Nearly freeness of reduced curves}
 Let $C$ be a reduced curve $\mathbb{P}^{2}_{\mathbb{C}}$ of degree $m$ given by $f \in S :=\mathbb{C}[x,y,z]$. We denote by $J_{f}$ the Jacobian ideal generated by the partial derivatives $\partial_{x}f, \, \partial_{y}f, \, \partial_{z}f$. Moreover, we denote by $r:={\rm mdr}(f)$ the minimal degree of a non-trivial relation among the partial derivatives, i.e., the minimal degree $r$ of a triple $(a,b,c) \in S_{r}^{3}$ such that 
$$a\cdot \partial_{x} f + b\cdot \partial_{y}f + c\cdot \partial_{z}f = 0.$$
We denote by $\mathfrak{m} = \langle x,y,z \rangle$ the irrelevant ideal. Consider the graded $S$-module $N(f) = I_{f} / J_{f}$, where $I_{f}$ is the saturation of $J_{f}$ with respect to $\mathfrak{m}$. 
\begin{definition}
A reduced plane curve $C$ is \emph{nearly free} if $N(f) \neq 0$ and for every $k$ one has ${\rm dim} \, N(f)_{k} \leq 1$.
\end{definition}

In order to study nearly freeness of conic-line arrangements, we will use \cite[Theorem 1.3]{Dimca1}.
\begin{theorem}[Dimca]
\label{Dim}
Let $\mathcal{C} \subset \mathbb{P}^{2}_{\mathbb{C}}$ be a conic-line arrangement of degree $m$ and let $f=0$ be its defining equation. Denote by $r: = {\rm mdr}(f)$. Then $\mathcal{C}$ is nearly free if and only if
\begin{equation}
\label{Milnor} 
r^2 - r(m-1) + (m-1)^2 = \tau(\mathcal{C})+1,
\end{equation}
where $\tau(\mathcal{C})$ is the total Tjurina number of $\mathcal{C}$.
\end{theorem}
\begin{remark}
In the original formulation of the above result there was the assumption that $r\leq m/2$. However, it turns out that it is not necessary, and this follows from \cite[Theorem 3.2]{duP}. 
\end{remark}
Now we are going to discuss the freeness from the homological viewpoint. We need the following result which comes from \cite{DimcaSticlaru}.
\begin{theorem}[Dimca-Sticlaru]
If $C\subset \mathbb{P}^{2}_{\mathbb{C}}$ is a nearly free reduced curve of degree $m$ given by $f \in S_{m}$, then the minimal resolution of the Milnor algebra $M(f)$ has the following form:
\begin{equation*}
\begin{split}
0 \rightarrow S(-d_{2}-m)\rightarrow S(-d_{1}-(m-1))\oplus S(-d_{2}-(m-1)) \oplus S(-d_{2}-(m-1)) \\ \rightarrow S^{3}(-m+1)\rightarrow S \rightarrow M(f) \rightarrow 0
\end{split}
\end{equation*} for some integers $d_{1} \leq d_{2}$ such that $d_{1} + d_{2} = m$. 
\end{theorem}
In the setting of the above theorem, the pair $(d_{1},d_{2})$ is called the set of exponents of the nearly free curve $C$.
\section{Nearly free arrangements of conic and lines with nodes, tacnodes,  and triple points}
In order to understand conic-line arrangements with nodes, tacnodes, and ordinary tiple points that are nearly free, we provide an upper-bound on the degree of such arrangements. We need the following result \cite[Proposition 4.7]{DimcaPokora}.
\begin{proposition}
\label{prop4.1}
Let $\mathcal{C} \, : \, f = 0$ be a conic-line arrangement of degree $m$ in $\mathbb{P}^{2}_{\mathbb{C}}$ such that it has only nodes, tacnodes, and ordinary triple intersection points. Then one has
$${\rm mdr}(f) \geq \frac{2}{3}m - 2.$$
\end{proposition}
If $\mathcal{C} \, : f=0$ is a nearly free conic-line arrangement of degree $m$ with nodes, tacnodes, and ordinary triple intersection points with the exponents $(d_{1}, d_{2})$, $d_{1}\leq d_{2}$, then ${\rm mdr}(f) = d_{1}$, and since 
$$2d_{1} \leq d_{1}+d_{2} =m$$
we obtain that ${\rm mdr}(f) \leq m/2$. Combining it with the above proposition, we arrive at
$$\frac{2}{3}m - 2 \leq {\rm mdr}(f) \leq m/2.$$
It gives us the following result.
\begin{proposition}
\label{bound}
If $\mathcal{C} \subset \mathbb{P}^{2}_{\mathbb{C}}$ is a nearly free conic-line arrangement of of degree $m$ with nodes, tacnodes, and ordinary triple intersection points, then $m\leq 12$. 
\end{proposition}

Based on the above proposition, we can formulate the following problem.

\begin{problem}
Classify all weak combinatorics of conic-line arrangements with nodes, tacnodes, and ordinary triple points in $\mathbb{P}^{2}_{\mathbb{C}}$ that are nearly free.
\end{problem}
Here by a weak combinatorics we mean the vector $(d, k;n_{2},t,n_{3})$, where $d$ is the number of lines, $k$ is the number of conics (and of course $m=2k+d$).

Let us pass to some combinatorial constraints on the singular points of such conic-line arrangements that come from the data of the exponents $(d_{1},d_{2})$. If $\mathcal{C}$ is a nearly free conic-line arrangement with $k$ conics and $d$ lines, $d_{1} + d_{2}=m$, then following equality holds
\begin{equation}
d^{2}_{1} + d_{2}^{2} + d_{1}d_{2}-d_{1}-2d_{2} = n_{2} + 3t + 4n_{3}.    
\end{equation}
Since $\binom{d_{1}+d_{2}}{2} - k = n_{2} + 2t + 3n_{3}$, we can obtain the following equality
$$t + n_{3} = d^{2}_{1} + d_{2}^{2} + d_{1}d_{2}-d_{1}-2d_{2} - \binom{d_{1}+d_{2}}{2}+k.$$
This gives
\begin{equation}
\label{tacn3}
2(t+n_{3}) = d_{1}^{2} + d_{2}^{2} - d_{1} - 3d_{2} + 2k.
\end{equation}
The last ingredient that we will use in our classification problem is the following proposition which is a direct consequence of the previously known results from \cite{DimcaPokora,Pokora}.
\begin{proposition}
\label{Hirz}
Let $\mathcal{C}$ be a conic-line arrangement in $\mathbb{P}^{2}_{\mathbb{C}}$ of degree $m=2k+d\geq 6$ having only $n_{2}$ nodes, $t$ tacnodes, and $n_{3}$ ordinary triple points. Then we have the following inequality
$$8k + n_{2} + \frac{3}{4}n_{3}\geq d + \frac{5}{2}t.$$
\end{proposition}
\begin{proof}
It follows from the discussions in the framework of \cite[Theorem 2.1]{DimcaPokora} and \cite[Theorem B]{Pokora}.
\end{proof}
\section{Partial classification of nearly free conic-line arrangements}
Here we perform a step-by-step approach towards the classification problem of our nearly free conic-line arrangements in the complex projective plane having $k\geq 1$ conics and $d\geq 1$ lines. We start with constructing explicit examples of nearly free conic-line arrangements having degree up to $7$. Then we present our argument standing behind the non-existence of nearly free conic-line arrangement with nodes, tacnodes, and ordinary triple points having degree $m \in \{10,11,12\}$. Unfortunately, our method does not allow us to decide on the non-existence of conic-line arrangements having degree $m \in \{8,9\}$, but we hope to resolve that problem by using a different approach.

\subsection{Case $m=3$}
Let us consider the following arrangement $\mathcal{C}_{3} = \{\ell, C\} \subset \mathbb{P}^{2}_{\mathbb{C}}$ defined by
$$F(x,y,z) = (x^{2} +y^{2} - 16z^{2})\cdot (y-x+4z).$$
It is easy to see that $\mathcal{C}_{3}$ has only $n_{2}=2$ and $t=n_{3}=0$, so its total Tjurina number $\tau(\mathcal{C}_{3})$ is equal to $2$. 
Using \verb}Singular}, we can compute ${\rm mdr}(F)$ which is equal to $1$.\\
By Theorem \ref{Dim} we see that
$$3= 1-2 + 4 = r^{2} - r(m-1)+(m-1)^{2} = \tau(\mathcal{C}_{3}) + 1 = 2+1,$$
so $\mathcal{C}_{3}$ is nearly free.
Observe that this is the only possible nearly free conic-line arrangement with nodes, tacnodes, ordinary triple points and $m=3$. Since $d_{1} + d_{2} = 3$ and $d_{1}\leq d_{2}$, we can have only $(d_{1},d_{2})=(1,2)$. Recall that ${\rm mdr}(f) = d_{1}=1$, and 
$$3=1^{2} - 1\cdot(3-1) + (3-1)^{2} =\tau(\mathcal{C})+1.$$
It means that $\tau(\mathcal{C})=2$, and the only possibility is to have $n_{2}=2$ and $t=0$, which completes our justification. 
\begin{figure}[ht]
\centering
\begin{tikzpicture}[line cap=round,line join=round,>=triangle 45,x=0.41613885779737725cm,y=0.4151073010090355cm,scale=0.6]
\clip(-9.372917867307349,-9.00970190076682) rectangle (24.881802226839987,8.087451162096906);
\draw [line width=2.pt] (0.,0.) ellipse (1.664555431189509cm and 1.660429204036142cm);
\draw [line width=2.pt,domain=-9.372917867307349:24.881802226839987] plot(\x,{(-4.--1.*\x)/1.});
\begin{scriptsize}
\draw [fill=black] (4.,0.) circle (2.0pt);
\draw [fill=black] (0.,-4.) circle (2.0pt);
\end{scriptsize}
\end{tikzpicture}
\caption{A nearly free arrangement with one conic and one line.}\label{rys1}
\end{figure}
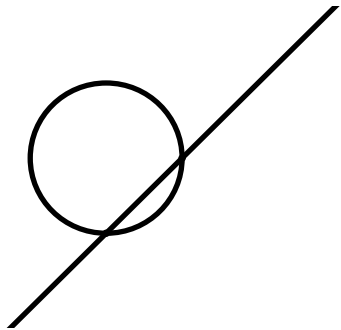

\subsection{Case $m=4$}
Now we consider the following arrangement $\mathcal{C}_{4} = \{\ell_{1},\ell_{2}, C\} \subset \mathbb{P}^{2}_{\mathbb{C}}$ given by 
$$F(x,y,z) = (x^{2} +y^{2} - 16z^{2})\cdot (y-x-4z)\cdot (y+x-4z).$$
It is easy to see that for $\mathcal{C}_{4}$ we have $n_{2}=2$ and $n_{3}=1$, so its total Tjurina number $\tau(\mathcal{C}_{4})$ is equal to $6$. Using \verb}Singular}, we can compute ${\rm mdr}(F)$ that is equal to $2$.\\
Using Theorem \ref{Dim}, we see that
$$7 = 4-6 +9 = r^{2} - r(m-1)+(m-1)^{2} = \tau(\mathcal{C}_{4}) + 1 = 6+1,$$
so $\mathcal{C}_{4}$ is nearly free.

\begin{figure}[ht]
\centering
\begin{tikzpicture}[line cap=round,line join=round,>=triangle 45,x=0.41613885779737725cm,y=0.4151073010090355cm, scale=0.6]
\clip(-9.372917867307349,-9.00970190076682) rectangle (24.881802226839987,8.087451162096906);
\draw [line width=2.pt] (0.,0.) ellipse (1.664555431189509cm and 1.660429204036142cm);
\draw [line width=2.pt,domain=-9.372917867307349:24.881802226839987] plot(\x,{(--4.--1.*\x)/1.});
\draw [line width=2.pt,domain=-9.372917867307349:24.881802226839987] plot(\x,{(--4.-1.*\x)/1.});
\begin{scriptsize}
\draw [fill=black] (4.,0.) circle (2.0pt);
\draw [fill=black] (0.,4.) circle (2.0pt);
\draw [fill=black] (-4.,0.) circle (2.0pt);
\end{scriptsize}
\end{tikzpicture}
\caption{A nearly free arrangement of degree $m=4$ with two nodes and one triple point.}\label{rys5}
\end{figure}
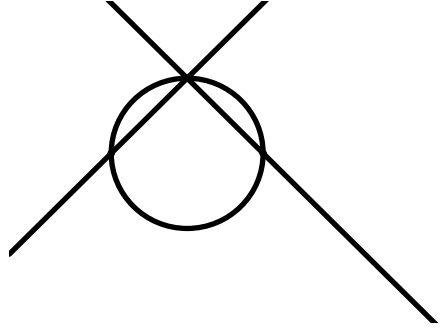
Now we show that for $m=4$ there is another possible nearly free conic-line arrangement. Observe that for $m=4$ we have two possible pairs of the exponents, namely $(d_{1},d_{2}) = (1,3)$ and $(d_{1},d_{2})=(2,2)$, and since $k\geq 1$, $d\geq 1$, our conic-line arrangement $\mathcal{C}$ consists of $k=1$ conics and $d=2$ lines. 
We see that for both possibilities of the exponents we have 
$\tau(\mathcal{C})=6$. The following weak combinatorics for conic line arrangements with nodes, tacnodes, and ordinary triple points and $m=4$ are admissible: 
\begin{center}
\begin{tabular}{||c c c||} 
 \hline
 $n_{2}$ & $t$ & $n_{3}$ \\ [0.5ex] 
 \hline\hline
 5 & 0  & 0   \\ 
 3 & 1  & 0   \\
 1 & 2  & 0  \\
 2 & 0  & 1  \\ [1ex] 
 \hline
\end{tabular}
\end{center}
Based on that, an arrangement with $3$ double points and one tacnode can be also nearly free.
Consider the arrangement $\mathcal{C}_{4}^{'} = \{\ell_{1},\ell_{2}, C\} \subset \mathbb{P}^{2}_{\mathbb{C}}$ given by 
$$G(x,y,z) = (x^{2} +y^{2} - 16z^{2})\cdot (y-4z)\cdot(y-x).$$
 Using \verb}Singular}, we can compute that ${\rm mdr}(G)$ is equal to $2$ and 
 $$7= 4 - 6 +9 = r^{2} - r(m-1)+(m-1)^{2} = \tau(\mathcal{C}_{4}^{'}) + 1 = 3 + 3\cdot 1 +1,$$
 \begin{figure}[ht]
\centering
\begin{tikzpicture}[line cap=round,line join=round,>=triangle 45,x=1.0cm,y=1.0cm, scale=0.28]
\clip(-10.82285310938766,-10.489844127057149) rectangle (12.28548981126729,7.241655604216725);
\draw [line width=2.pt] (0.,0.) circle (4.cm);
\draw [line width=2.pt,domain=-10.82285310938766:12.28548981126729] plot(\x,{(--4.-0.*\x)/1.});
\draw [line width=2.pt,domain=-10.82285310938766:12.28548981126729] plot(\x,{(-0.--1.*\x)/1.});
\end{tikzpicture}
\caption{A nearly free arrangement of degree $m=4$ with three nodes and one tacnode.}\label{rys5}
\end{figure}
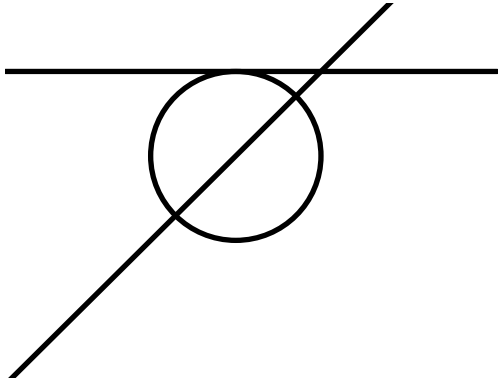
 so $\mathcal{C}_{4}^{'}$ is nearly free.
 
\subsection{Case $m=5$}
We consider here the following arrangement $\mathcal{C}_{5} = \{\ell_{1},\ell_{2},\ell_{3}, C\} \subset \mathbb{P}^{2}_{\mathbb{C}}$ given by 
$$F(x,y,z) = (x^{2} +y^{2} - 16z^{2})\cdot (y-x+4z)\cdot (y+x-4z)\cdot (y-x-4z).$$
It is easy to see that for $\mathcal{C}_{5}$ we have $n_{2}=3$ and $n_{3}=2$, so its total Tjurina number $\tau(\mathcal{C}_{5})$ is equal to $11$. Using \verb}Singular}, we can compute ${\rm mdr}(F)$ which is equal to $2$.\\
Using Theorem \ref{Dim}, we have
$$12= 4-8+16 = r^{2} - r(m-1)+(m-1)^{2} = \tau(\mathcal{C}_{5}) + 1 = 11+1,$$
so $\mathcal{C}_{5}$ is nearly free.

\begin{figure}[ht]
\centering
\begin{tikzpicture}[line cap=round,line join=round,>=triangle 45,x=0.4161388577973773cm,y=0.4151073010090355cm,scale=0.6]
\clip(-9.372917867307349,-9.00970190076682) rectangle (24.881802226839998,8.087451162096906);
\draw [line width=2.pt] (0.,0.) ellipse (1.6645554311895092cm and 1.660429204036142cm);
\draw [line width=2.pt,domain=-9.372917867307349:24.881802226839998] plot(\x,{(--4.--1.*\x)/1.});
\draw [line width=2.pt,domain=-9.372917867307349:24.881802226839998] plot(\x,{(--4.-1.*\x)/1.});
\draw [line width=2.pt,domain=-9.372917867307349:24.881802226839998] plot(\x,{(-4.--1.*\x)/1.});
\begin{scriptsize}
\draw [fill=black] (4.,0.) circle (2.0pt);
\draw [fill=black] (0.,4.) circle (2.0pt);
\draw [fill=black] (-4.,0.) circle (2.0pt);
\draw [fill=black] (0.,-4.) circle (2.0pt);
\end{scriptsize}
\end{tikzpicture}
\caption{A nearly free arrangement with one conic and three lines.}\label{rys3}
\end{figure}
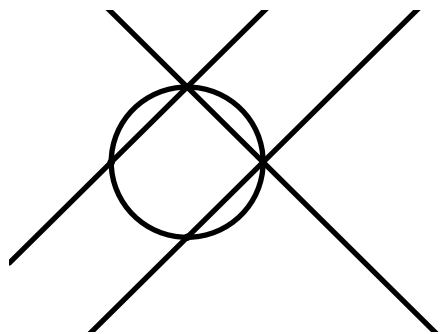

\subsection{Case $m=6$}
We consider here the following arrangement $\mathcal{C}_{6} = \{\ell_{1},\ell_{2},\ell_{3},\ell_{4}, C\} \subset \mathbb{P}^{2}_{\mathbb{C}}$ given by
$$F(x,y,z) = (x^{2} +y^{2} - 16z^{2})\cdot (y-x+4z)\cdot (y+x-4z)\cdot (y-x-4z)\cdot (y+x+4z).$$
It is easy to see that for $\mathcal{C}_{3}$ we have $n_{2}=2$ and $n_{3}=4$, so its total Tjurina number $\tau(\mathcal{C}_{6})$ is equal to $18$. Using \verb}Singular}, we can compute ${\rm mdr}(F)$ which is equal to $2$.\\
Using Theorem \ref{Dim}, we obtain
$$19= 4-10+25= r^{2} - r(m-1)+(m-1)^{2} = \tau(\mathcal{C}_{6}) + 1 = 18+1,$$
so $\mathcal{C}_{6}$ is nearly free.

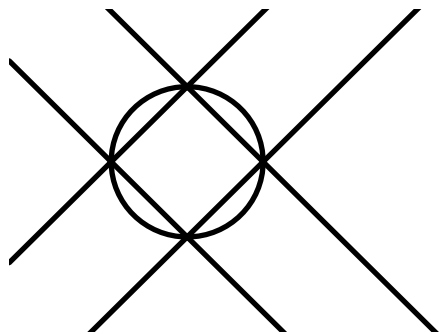
\begin{figure}[ht]
\centering
\begin{tikzpicture}[line cap=round,line join=round,>=triangle 45,x=0.41613885779737725cm,y=0.4151073010090355cm,scale=0.6]
\clip(-9.372917867307349,-9.00970190076682) rectangle (24.881802226839987,8.087451162096906);
\draw [line width=2.pt] (0.,0.) ellipse (1.664555431189509cm and 1.660429204036142cm);
\draw [line width=2.pt,domain=-9.372917867307349:24.881802226839987] plot(\x,{(--4.--1.*\x)/1.});
\draw [line width=2.pt,domain=-9.372917867307349:24.881802226839987] plot(\x,{(--4.-1.*\x)/1.});
\draw [line width=2.pt,domain=-9.372917867307349:24.881802226839987] plot(\x,{(-4.--1.*\x)/1.});
\draw [line width=2.pt,domain=-9.372917867307349:24.881802226839987] plot(\x,{(-4.-1.*\x)/1.});
\begin{scriptsize}
\draw [fill=black] (4.,0.) circle (2.0pt);
\draw [fill=black] (0.,4.) circle (2.0pt);
\draw [fill=black] (-4.,0.) circle (2.0pt);
\draw [fill=black] (0.,-4.) circle (2.0pt);
\end{scriptsize}
\end{tikzpicture}
\caption{A nearly free arrangement with one conic and four lines.}\label{rys4}
\end{figure}
\subsection{Case $m=7$}
Let us consider here the following arrangement $\mathcal{C}_{7} = \{\ell_{1},\ell_{2},\ell_{3},\ell_{4},\ell_{5}, C\} \subset \mathbb{P}^{2}_{\mathbb{C}}$ given~by
$$F(x,y,z) = (x^{2} +y^{2} - z^{2})\cdot (x^{2}-z^{2})\cdot (y^{2}-z^{2})\cdot (y+x).$$
It is easy to see that for $\mathcal{C}_{7}$ we have $n_{2}=6$, $t=4$, and $n_{3}=2$, so its total Tjurina number $\tau(\mathcal{C}_{7})$ is equal to $26$. Using \verb}Singular}, we can compute ${\rm mdr}(F)$ that is equal to $3$.\\
Using Theorem \ref{Dim}, we see 
$$27= 9-18+36= r^{2} - r(m-1)+(m-1)^{2} = \tau(\mathcal{C}_{7}) + 1 = 26+1,$$
so $\mathcal{C}_{7}$ is nearly free.

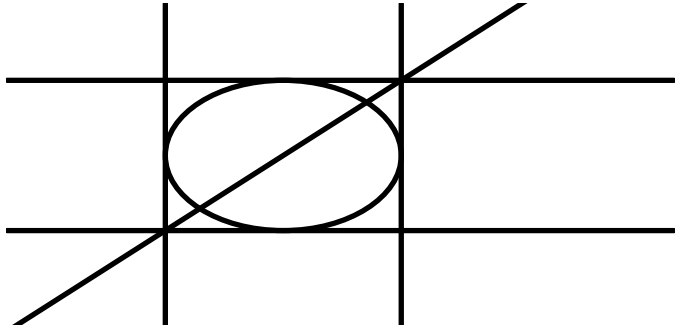
\begin{figure}[ht]
\centering
\begin{tikzpicture}[line cap=round,line join=round,>=triangle 45,x=0.6464093269651405cm,y=0.4151073010090362cm, scale=0.6]
\clip(-9.37291786730735,-9.009701900766835) rectangle (13.252113305987494,8.087451162096912);
\draw [line width=2.pt] (0.,0.) ellipse (2.585637307860562cm and 1.6604292040361448cm);
\draw [line width=2.pt] (-4.,-9.009701900766835) -- (-4.,8.087451162096912);
\draw [line width=2.pt,domain=-9.37291786730735:13.252113305987494] plot(\x,{(-32.-0.*\x)/-8.});
\draw [line width=2.pt] (4.,-9.009701900766835) -- (4.,8.087451162096912);
\draw [line width=2.pt,domain=-9.37291786730735:13.252113305987494] plot(\x,{(-16.-0.*\x)/4.});
\draw [line width=2.pt,domain=-9.37291786730735:13.252113305987494] plot(\x,{(-0.--8.*\x)/8.});
\begin{scriptsize}
\draw[color=black] (-3.694004835826135,7.740070843681836);
\end{scriptsize}
\end{tikzpicture}
\caption{A nearly free arrangement with one conic and five lines.}\label{rys6}
\end{figure}
\newpage
\subsection{Case $m=10$.}
Assume  that $\mathcal{C} \subset \mathbb{P}^{2}_{\mathbb{C}}$ is a nearly free conic-line arrangement with $m=10$ having the exponents $(d_{1},d_{2})$, $d_{1}\leq d_{2}$. By Proposition \ref{prop4.1}, we have
$$d_{1} = {\rm mdr}(f) \geq \frac{2}{3}m-2 = \frac{20}{3}-2 = 4 \,\frac{2}{3},$$
which implies that $d_{1}\geq 5$. However, it means that we need to consider the only one case, namely $(d_{1},d_{2})=(5,5)$.
Using equation \eqref{tacn3} we obtain
$$t+n_{3} = 15+k.$$
and it means, in particular, that $k \in \{1,2,3,4\}$.\\
By the combinatorial count, we have
$$\binom{10}{2}-k = n_{2} + 2(t+n_{3}) + n_{3},$$
so we arrive at
$$15-3k = n_{2}+n_{3},$$
and it means, in particular, that $k \in \{1,2,3,4\}$.
Hence
$$t=15+k-n_{3} = 15+k-(15-3k-n_{2}) = 4k+n_{2}.$$
Now we are going plug this data into inequality from Proposition \ref{Hirz}.\\
We have
$$8k+n_{2} + \frac{3}{4}(15-3k-n_{2}) \geq (10-2k)+\frac{5}{2}(4k+n_{2}).$$
After some simple manipulations we obtain
$$5\geq 9k + 9n_{2},$$
which is a contradiction since $k \in \{1,2,3,4\}$.
This proves the following result.

\begin{theorem}
There does not exists any nearly free conic-line arrangement in the complex projective plane with nodes, tacnodes, and ordinary triple points having degree $m=10$.
\end{theorem}

\subsection{Case $m=11$.}
Assume  that $\mathcal{C} \subset \mathbb{P}^{2}_{\mathbb{C}}$ is a nearly free conic-line arrangement with $m=11$ having the exponents $(d_{1},d_{2})$, $d_{1}\leq d_{2}$. By Proposition \ref{prop4.1}, we have that
$$d_{1} = {\rm mdr}(f) \geq \frac{2}{3}m-2 = \frac{22}{3}-2 = 5 \,\frac{1}{3},$$
so $d_{1} \geq 6$. However, it means that such a nearly free curve cannot exists, and we have the following proposition.

\begin{proposition}
There does not exists any nearly free conic-line arrangement in the complex projective plane with nodes, tacnodes, and ordinary triple points having degree $m=11$.
\end{proposition}

\subsection{Case $m=12$.}
Assume that $\mathcal{C} \subset \mathbb{P}^{2}_{\mathbb{C}}$ is a nearly free conic-line arrangement with $m=12$ having the exponents $(d_{1},d_{2})$ with $d_{1}\leq d_{2}$. Using Proposition \ref{prop4.1}, we see that
$$d_{1}={\rm mdr}(f) \geq \frac{2}{3}m-2 = \frac{2}{3}\cdot 12 - 2 = 6,$$
so the only one case to consider is $(d_{1},d_{2})=(6,6)$.
Using equation \eqref{tacn3} we obtain
$$2(t+n_{3}) = 48 + 2k,$$
and plugging this into the combinatorial count we get
$$n_{2}+n_{3}=18-3k.$$
In particular, $k \in \{1,2,3,4,5\}$. \\
Using Proposition \ref{Hirz} we have
$$5k+18 \geq d + \frac{5}{2}t,$$ 
so finally
$$t \leq \frac{2}{5}\bigg(5k+18-d\bigg).$$
We have the following possibilities, depending on $k$, namely
\begin{table}[h]
\centering
\begin{tabular}{||c c c||} 
 \hline
 $k$ & $n_{3}$ & $t \leq$ \\ [0.5ex] 
 \hline\hline
 1 & 15  & 5   \\ 
 2 & 12  & 8   \\
 3 & 9   & 11  \\
 4 & 6   & 14  \\ 
 5 & 3   & 16    \\[1ex] 
 \hline
\end{tabular}.
\end{table} 
\newpage
Since in all the cases listed above we have $t+n_{3}\leq 20$, we arrive at a contradiction with respect to the condition that $t+n_{3} = 24+k$.
This allows us to conclude our discussion by the following result.

\begin{theorem}
There does not exists any nearly free conic-line arrangement in the complex projective plane with nodes, tacnodes, and ordinary triple points having degree $m=12$.
\end{theorem}
\section*{Acknowledgments}
This note is a part of the author's Master Thesis that is written under the supervision of Piotr Pokora. The research was conducted in the framework of \textit{The Excellence Initiative - Research University} Programme at the Pedagogical University of Cracow.

The author would like to warmly thank an anonymous referee for very useful suggestions and remarks that allowed to improve the paper,  especially for explaining the non-existence case with $m=10$.

\vskip 0.5 cm

Aleksandra Ga\l ecka \\
Department of Mathematics,
Pedagogical University of Krak\'ow,
Podchor\c a\.zych 2,
PL-30-084 Krak\'ow, Poland. \\
\nopagebreak
\textit{E-mail address:} \texttt{aleksandra.galecka@student.up.krakow.pl}
\bigskip
\end{document}